%% file: B-Chaumont.tex
\documentclass[final,12pt]{amsart}

\usepackage[pages=all, color=black, position={current page.south}, placement=bottom, scale=1, opacity=1, vshift=5mm]{background}

\usepackage[margin=1in]{geometry} 

\usepackage{amssymb}
\usepackage{amsmath}
\usepackage{amsmath}
\usepackage{amsthm}
\usepackage{amssymb}
\usepackage{float}
\usepackage{multirow}
\usepackage{tikz}
\usepackage{enumerate}
\usepackage[caption=false]{subfig}
\usepackage[utf8]{inputenc}
\usepackage{hyperref}
\usepackage{algorithm}
\usepackage[export]{adjustbox} 

\usepackage{oubraces}
\usepackage{stix}
\usepackage{stmaryrd}
\theoremstyle{plain}
\newtheorem{theorem}{Theorem}[section]

\newtheorem{lemma}[theorem]{Lemma}

\theoremstyle{definition}

\newtheorem{Remark}[theorem]{Remark}

\newcommand\bR{{\mathbb{R}}}

\usepackage{graphicx, color}

\input{general_commands}

\newcommand{\ts}{\widetilde s}

\usepackage{algorithm, algpseudocode} 
\usepackage{mathrsfs} 

\usepackage{lipsum}

\author{
  Gabriel R.  Barrenechea
  }
  \address{
  Gabriel R.  Barrenechea
  \thanks{
    Department of Mathematics and Statistics,
    University of Strathclyde, 26 Richmond Street, Glasgow G1 1XK, UK,
    {\tt{gabriel.barrenechea@strath.ac.uk}}.
}
}

\author{
  Th\'eophile Chaumont-Frelet
}
\address{
  Th\'eophile Chaumont-Frelet
  \thanks{
    Inria Univ.  Lille and Laboratoire Paul Painlev\'e, 59655 Villeneuve-d’Ascq, France,
    {\tt{theophile.chaumont@inria.fr}}.
}}

\title[Nitche's trick for positivity preservation]{An Aubin-Nitsche Lemma for a positivity-preserving finite element method}

\date{\today}

\begin{document}

\maketitle

\begin{abstract}
  In this work we prove an Aubin-Nitsche Lemma for a positivity-preserving discretisation of an elliptic
  problem.  Due to the nonlinearity of the discretisation, the result requires as a first step
  the proposal of a linearised adjoint problem that can be linked to the method of
  choice by appropriately selecting weights.  This linearised adjoint problem, together 
  with a regularity result, allow the proof of an optimal-order error estimate in the $L^2$-norm of
  the error.
\end{abstract}


\section{Introduction}
\label{sec:intro}

The study of positivity-preserving discretisations, at least in the context of finite element methods, started with the seminal work \cite{CR73}.  From very early on it was understood that for a linear discretisation to
preserve the Discrete Maximum Principle (DMP), or satisfy the weaker
condition of preserving positivity (or previously set bounds),  the mesh needs to satisfy stringent
conditions (especially if the problem includes convection), and the method needs to be of low order
(at most first order polynomial degree). This motivated the subsequent development of numerous
nonlinear finite element methods which achieve the desired properties by adding appropriate
stabilising terms to the formulation, or by writing the problem as a variational inequality and add
Lagrange multipliers.  A thorough literature review is out of
the scope of this work, and we will only cite the monographs \cite{BJK25} for elliptic, and \cite{KH23} for 
hyperbolic partial differential equations.  
 
 Despite the vast amount work devoted to proposing and analysing nonlinear schemes that
respect the DMP/satisfy bounds, there are two
topics that, up to the best of our knowledge,  remain open. The first one is the possibility of localising the
error estimates (in the spirit of \cite{JSW87}). The second one
is the possibility of proving optimal order error estimates in the $L^2(\Omega)$-norm. 
This work is devoted to the latter of these open problems.  The nonlinearity of a typical
positivity-preserving discretisation makes the use of the Aubin-Nitsche Lemma more involved,
as the concept of an adjoint problem becomes somewhat more problematic. In this work,
we follow an approach described in~\cite[Section 3.3.4]{vanderzee2009}
that consists in proposing a linearised adjoint problem.
We apply this idea to  the positivity-preserving scheme recently proposed in
\cite{BGPV24} that, instead of  adding a nonlinear stabilising term to the formulation,  seeks directly
for a finite element function that belongs to the \textit{admissible set} of finite element functions that
satisfy the prescribed bounds at their degrees of freedom.  The analysis of this method has usually been carried out
by rewriting the method as a variational inequality.   This last point also makes this result interesting, as 
whether the Aubin-Nitsche can be extended to variational inequalities also seems to be an open question.  

Due to technical difficulties, we carry out the analysis for the one-dimensional case and the piecewise
linear finite element space.  This work can be perceived then as the first step towards developing a full
theory of optimal order $L^2$-error estimates for nonlinear, positivity-preserving discretisations.  The rest
of the manuscript is organised as follows: we present the model problem and the finite element method
used in Section~\ref{Sec:Prelim}, and the linearised adjoint problem and the main result of this
work are presented in Section~\ref{Sec:main}.

\section{The model problem and preliminaries}\label{Sec:Prelim}

Throughout this note $\Omega \subset \mathbb R$ denotes an open interval.
We will adopt standard notation for Sobolev spaces, aligned with, e.g., \cite{Brezis-11}. For example, 
$L^2(\Omega)$ and $H^1_0(\Omega)$ are the standard Lebesgue and Sobolev
spaces, and if $v,w \in L^2(\Omega)$ and $\gamma>0$, 
\begin{equation*}
(v,w)_\Omega := \int_\Omega vw\,\textrm{d}x\quad\textrm{and}\quad\|v\|_\Omega \eq \left\{\int_\Omega |v|^2\textrm{d}x\right\}^{1/2}\quad\textrm{and}\quad \|v\|_{\gamma,\Omega} := \gamma^{1/2}\|v\|_\Omega\,,
\end{equation*}
denote the inner product, and standard and weighted  norms in $L^2(\Omega)$, respectively.  

\subsection{The model problem} 
Let $f \in L^2(\Omega)$,   $\mu > 0$ and $\beta \in\bR$.  The model problem reads as follows:
find $u \in H^1_0(\Omega)$ such that
\begin{equation}\label{Eq:Model}
a(u,v) = (f,v)_\Omega \quad \forall v \in H^1_0(\Omega).
\end{equation}
Here,   the bilinear form $a(\cdot,\cdot)$ is given by
\begin{equation*}
a(\phi,v) \eq \mu(\phi,v)_\Omega + \beta(\phi',v)_\Omega+ (\phi',v')_\Omega.
\end{equation*}
This problem has a unique solution thanks to Lax-Milgram's Lemma (see, e.g., \cite[Lemma~25.2]{EG21-II}).
The assumption that $\mu>0$ has been done for simplicity. It is possible to extend the results proven
in Section~\ref{Sec:main} to the case $\mu=0$ without any essential difficulty.

If the right-hand side $f$ is non-negative in $\Omega$, then thanks to the maximum principle (see,
e.g.,  \cite{renardy2006}), we can  make the following assumption on the solution of \eqref{Eq:Model}: 
\begin{equation}\label{A1}
u(x) \in [0,\infty) \qquad \forall x \in \Omega\,.
\end{equation}
In addition, this assumption can be interpreted  pointwise due  to the injection
$H^1_0(\Omega) \hookrightarrow C^0(\overline{\Omega})$ in one space dimension (see, e.g., \cite{Brezis-11}).

\subsection{The discrete setting}

We consider a mesh $\CT_h$ of $\Omega$.  We denote by $h_K\eq \textrm{diam}(K)$, and $h:=\max\{h_K:K\in
\CT_h\}$. To avoid technical complications, we will assume that the mesh family $\{\CT_h\}_{h>0}$ is quasi-uniform. Associated to each mesh in the family, we denote by 
\begin{equation*}
H^2(\CT_h) :=\, \{v \in L^2(\Omega):v|_K \in H^2(K)\,\textrm{ for all}\,K \in \CT_h\}\,,\\
\end{equation*}
with seminorm
\begin{equation*}
\|v''\|_{\CT_h}^2 \eq\, \sum_{K \in \CT_h} \|(v|_K)''\|_K^2\,.
\end{equation*}

Moreover, we introduce the piecewise linear finite element space
\begin{equation*}
V_h \eq \left \{
v_h \in H^1_0(\Omega) \; | \; v_h|_K \in \CP_1(K) \; \forall K \in \CT_h
\right \}.
\end{equation*}
The standard Lagrangian basis for $V_h$, that is, the one containing the ``hat'' functions, is denoted
by $\{\Psi_j\}_{j=1}^N$.
Using this basis,  the  Lagrange interpolant $I_h: H^1_0(\Omega) \to V_h$ is defined by
\begin{equation*}
I_h v \eq \sum_{j=1}^N v(x_j) \Psi_j.
\end{equation*}
The Lagrange interpolant is an algebraic projection onto $V_h$, and it satisfies the
following approximation property (see, e.g., \cite{EG21-I}): there exists $C>0$, independent of $h$, such that
\begin{equation}
\label{eq_interpolation_error}
h^{-1}\|v-I_h v\|_\Omega+\|(v-I_h v)'\|_\Omega
\leq
C h \|v''\|_{\CT_h}\qquad\forall\, v \in H^1_0(\Omega) \cap H^2(\CT_h)\,.
\end{equation}

\subsection{The finite element scheme} 
We start defining, for all $v_h \in V_h$,  the \textit{constrained} $v_h^+$ and \textit{complementary}
parts as 
\begin{equation*}
v_h^+ \eq \sum_{j=1}^N v_h(x_j)^+ \Psi_j \in V_h\quad\textrm{and}\quad v_h^-=v_h-v_h^+\,,
\end{equation*}
where  $v_h(x_j)^+:=\max\{0, v_h(x_j)\}$.

With the above notations,  the finite element method proposed in \cite{BGPV24}  reads as follows: find $u_h \in V_h$ such that
\begin{equation}\label{FEM:BPM}
a(u_h^+,v_h) + s_h(u_h^-,v_h) = (f,v_h)_\Omega\qquad\forall\, v_h \in V_h\,.
\end{equation}
Here, the \textrm{stabilising} bilinear form $s(\cdot,\cdot)$ is given by
\begin{equation*}
s_h(\phi,v) \eq \alpha_h \sum_{j=1}^N \phi(x_j)v(x_j)\quad\textrm{where}\quad \alpha_h \eq \mu h + \beta + h^{-1}\,.
\end{equation*}

\begin{Remark}
Since the problem we are analysing in this work includes convection, then we should refer to \cite{ABP24} for
its proposal and analysis.  Although, the techniques used to analyse the convection-diffusion problem
differ from those used in \cite{BGPV24}, they share the main arguments used herein.
\end{Remark}

\section{The $L^2(\Omega)$-error estimate}\label{Sec:main}

As it was mentioned in the introduction, the method \eqref{FEM:BPM} was analysed in \cite{BGPV24}, where
optimal convergence is proven in the ``energy'' norm, but no optimal order error estimate in the
$L^2(\Omega)$-norm was proven, even if it has been observed in the numerical experiments.
The first step towards the proof  is to introduce the following weights:
for $1 \leq j \leq N$,  define
\begin{equation}\label{Eq:Weights}
\omega_j \eq -\frac{u_h(x_j)^-}{u(x_j)-u_h(x_j)^+}, 
\end{equation}
with the convention that ``$0/0 = 0$''.

\begin{lemma}[Weights positivity]
For $1 \leq j \leq N$, either $u(x_j) = u_h(x_j)^+ = 0$, or, we have $0 \leq \omega_j < +\infty$.
\end{lemma}

\begin{proof}
Fix $j \in \{1,\dots,N\}$. If $u_h(x_j)^- = 0$, then $\omega_j = 0$ and there is nothing to check.
We can therefore assume that $u_h(x_j)^- < 0$.  If $u(x_j) = u_h(x_j)^+ = 0$, there is nothing to check.  Since $u(\cdot)$ is non-negative
we can therefore assume that $u(x_j) > 0$.  Then,  $u(x_j) > u_h(x_j)^+ = 0$, and so $0<\omega_j=-u_h(x_j)^-/u(x_j)<\infty$, which proves the claim.
\end{proof}

\subsection{The adjoint problem}
In order  to define the adjoint problem that will be at the core of the proof
of our main result we define the following set:
\begin{equation*}
J_0:=\{ j\in\{1,\ldots,N\}: \omega_j = \infty\}=\{ j\in\{1,\ldots,N\}: u(x_j) = u_h(x_j)^+ = 0 \;\textrm{and}\;u_h(x_j)^- < 0 \}\,.
\end{equation*}
 Note that in practice the indices are contained in $J_0$ are not known, but this is not an issue since
this set is only used as an abstract tool to be used in the proof of the error estimate.
We also denote by $\{I^\ell\}_{\ell=1}^L$ the partition of $\Omega$ into (open)
intervals induced by the nodes $x_j$ with $j \in J_0$. Finally, we let
\begin{equation*}
\LH_0
=
\left \{
v \in H^1_0(\Omega)\; | \; v(x_j) = 0 \; \forall j \in J_0
\right \}.
\end{equation*}

In addition, we introduce the following bilinear form 
\begin{equation*}
\ts_h(\phi,v)
\eq
\alpha_h \sum_{j=1}^N \omega_j \phi(x_j)v(x_j)\quad \textrm{for}\; \phi,v \in \LH_0.
\end{equation*}
Observe that this bilinear form is well-defined since $\phi(x_j) = v(x_j) = 0$
whenever $\omega_j = +\infty$, and since the remaining weights are
all positive, $\ts_h$ is a continuous bilinear form over $\LH_0$, and
we have
\begin{equation}
\label{eq_tsh_positive}
\ts_h(\phi,\phi) \geq 0
\end{equation}
for all $\phi \in \LH_0$.  Moreover, the following property will be instrumental in the proof of the error estimate:
\begin{equation}
\label{eq_tsh_sh}
\ts_h(u-u_h^+,v) = -s_h(u_h^-,v)
\end{equation}
for all $v \in \LH_0$.

\begin{lemma}[Adjoint problem]
\label{lemma_theta}
Let $\psi \in L^2(\Omega)$. Then, the problem: find $\theta \in \LH_0$ such that
\begin{equation}
\label{eq_definition_theta}
a(w,\theta) + \ts_h(w,\theta) = \mu (w,\psi)_\Omega\qquad\forall\, w \in \LH_0\,,
\end{equation}
has a unique solution. In addition, it satisfies
\begin{equation} \label{eq_theta_H1}
\|\theta\|_{\mu,\Omega} \leq \|\psi\|_{\mu,\Omega}
\quad\textrm{and}\quad
\|\theta'\|_{\Omega} \leq\|\psi\|_{\mu,\Omega}.
\end{equation}
\end{lemma}

\begin{proof} 
Due to the injection $H^1_0(\Omega) \hookrightarrow C^0(\overline{\Omega})$
and the fact the functions in $\LH_0$ vanish at the nodes $x_j$ with infinite
weights $\omega_j$ (i.e. $j \in J_0$), it is clear that $\ts_h$ is a continunous
bilinear form over $\LH_0$. We also noted in \eqref{eq_tsh_positive} that $\ts_h$
is positive semi-definite. It follows that the bilinear form
\begin{equation*}
\phi,v \in \LH_0 \mapsto \mathbb R \ni a(\phi,v) + \ts_h(\phi,v)
\end{equation*}
is continuous and elliptic, and the existence and uniqueness of $\theta \in \LH_0$
satisfying \eqref{eq_definition_theta} for any $\psi \in L^2(\Omega)$ follows from the Lax-Milgram Lemma. 

To prove the first inequality in \eqref{eq_theta_H1}, we write that
\begin{equation*}
\|\theta\|_{\mu,\Omega}^2 \leq a(\theta,\theta) \leq a(\theta,\theta) + \ts_h(\theta,\theta)
=
(\mu \theta,\psi)_\Omega,
\end{equation*}
and the estimate follows using the Cauchy-Schwarz inequality. For the second part of \eqref{eq_theta_H1}
we similarly write that
\begin{equation*}
\|\theta'\|_\Omega^2 \leq a(\theta,\theta) \leq (\mu\theta,\psi)_\Omega \leq \|\psi\|_{\mu,\Omega}^2\,,
\end{equation*}
where we used the first part in the last inequality.
\end{proof}

\subsection{The discrete Green's function}

One further ingredient that is needed to prove the $L^2(\Omega)$-error estimate is the \textit{discrete
Green's function}.  At each node $x_j$ of the mesh, with $j \notin J_0$, this  is defined as the only element 
$\phi_{j,h} \in V_h \cap \LH_0$ that satisfies
\begin{equation}\label{Eq:discrete_green}
(w',\phi_{j,h}')_\Omega = w(x_j)\qquad \forall\, w \in \LH_0\,.
\end{equation}

The solution, $\theta$, to the adjoint problem \eqref{eq_definition_theta} cannot, in general,   be proven 
to belong to $H^2(\Omega)$. Nevertheless, with
the help of the Green's function just defined a decomposition of $\theta$ in terms of a sum of finite element functions
and a regular function can be proven. This is done in the next result.

\begin{lemma}[Regularity by decomposition]
\label{lemma_theta_H2}
Consider $\psi \in L^2(\Omega)$ and let $\theta \in \LH_0$ solve \eqref{eq_definition_theta}.
Then, the following decomposition holds: $\theta = \theta_2+\theta_h$, where $\theta_h$ is a linear
combination of discrete Green's functions, and  $\theta_2 \in H^1_0(\Omega)$ satisfying $\theta_2 \in H^2(I_\ell)$ for $\ell=1,\ldots,L$.  In addition, 
$\theta_2$ satisfies
\begin{equation}
\label{eq_theta_H2}
\|\theta_2''\|_{\CT_h}
\leq
(2\,\mu^{1/2}+\beta) \|\psi\|_{\mu,\Omega}.
\end{equation}
\end{lemma}

\begin{proof}
We start rewriting \eqref{eq_definition_theta} as
\begin{equation*}
(w',\theta')_\Omega = \mu(w,\psi)_\Omega-\mu(w,\theta)_\Omega-\beta (w',\theta)_\Omega-\ts_h(w,\theta)\qquad \forall\, w \in \LH_0\,.
\end{equation*}
 Applying  \eqref{Eq:discrete_green} at each node $j\notin J_0$, we conclude that
there exists $\theta_h \in V_h$ such that
\begin{equation*}
(w',\theta_h')_\Omega = -\ts_h(w,\theta) = -\alpha_h \sum_{j=1}^N \omega_j w(x_j) \theta(x_j)\qquad\forall\, w \in \LH_0\,.
\end{equation*}
We then introduce $\theta_2 \eq \theta-\theta_h \in \LH_0$ and observe that
\begin{equation*}
(w',\theta_2')_\Omega = (w,\mu \psi-\mu\theta+\beta \theta')_\Omega\qquad\forall\, w \in \LH_0\,.
\end{equation*}
It follows that $\theta_2 \in H^2(I^\ell)$ for
each $1 \leq \ell \leq L$ with
\begin{equation}
\label{tmp_theta_2}
(\theta_2|_{I^\ell})'' = - \left (\mu\psi-\mu\theta+\beta\theta'\right )|_{I^\ell}.
\end{equation}
Since each $K \in \CT_h$ is entirely included in some $I^\ell$, it follows in particular
that $\theta_2 \in H^2(\CT_h)$ with~\eqref{tmp_theta_2} holding on each $K \in \CT_h$.
As a consequence, we have
\begin{equation*}
\|\theta_2''\|_{\CT_h}
\leq
\mu^{1/2}(\|\psi\|_{\mu,\Omega}+\|\theta\|_{\mu,\Omega}) + \beta\|\theta'\|_\Omega,
\end{equation*}
and thus \eqref{eq_theta_H2} follows from the estimates in~\eqref{eq_theta_H1}.
\end{proof}

\subsection{The Aubin-Nitsche Lemma}
In this section we prove the main result of the present work. 

\begin{theorem}[Aubin-Nitsche trick]
There exists a constant $C>0$, independent of $h$ and any physical parameter, such that
\begin{equation}
\label{eq_aubin_nitsche}
\|u-u_h^+\|_{\mu,\Omega}
\leq
C \left (
\mu^{1/2}+\beta
\right )
h
\left (
\mu^{1/2} h \|u-u_h^+\|_{\mu,\Omega}
+
(1+\beta h)\|(u-u_h^+)'\|_{\Omega}
\right )\,.
\end{equation}
\end{theorem}

\begin{proof}
Let $\xi$ be the only element of $\LH_0$ such that
\begin{equation}\label{Adjoint-Final}
a(w,\xi)+\ts_h(w,\xi) = \mu(w,u-u_h^+)_\Omega\qquad w \in \LH_0\,.
\end{equation}
As shown in Lemma \ref{lemma_theta},
$\xi$ is  well-defined, and Lemma \ref{lemma_theta_H2} shows
that $\xi = \xi_h+\xi_2$ with $\xi_h \in V_h$ and $\xi_2 \in H^1_0(\Omega) \cap H^2(\CT_h)$
satisfying
\begin{equation}
\label{tmp_xi_H2}
\|\xi_2''\|_{\CT_h}
\leq
(2\,\mu^{1/2}+\beta)\|u-u_h^+\|_{\mu,\Omega}.
\end{equation}

Since $u-u_h^+ \in \LH_0$,  then it can be used as test function in \eqref{Adjoint-Final},  which, together with
 \eqref{eq_tsh_sh}  yield
\begin{equation*}
\|u-u_h^+\|_{\mu,\Omega}^2
=
a(u-u_h^+,\xi)+\ts_h(u-u_h^+,\xi)
=
a(u,\xi)-\{a(u_h^+,\xi)+s_h(u_h^-,\xi)\}\,.
\end{equation*}
Since $u_h$ solves \eqref{FEM:BPM} and $u$ solves \eqref{Eq:Model}, then 
the Galerkin orthogonality follows  in the sense that
\begin{equation*}
a(u,v_h)= (f,v_h)_\Omega=a(u_h^+,v_h)+s(u_h^-,v_h)\qquad\forall\, v_h\in V_h\,,
\end{equation*}
and so
\begin{equation*}
\|u-u_h^+\|_{\mu,\Omega}^2
=
a(u,\xi-I_h\xi)-\{a(u_h^+,\xi-I_h\xi)+s_h(u_h^-,\xi-I_h\xi)\}.
\end{equation*}
The next step is to observe that $(\xi-I_h\xi)(x_j) = 0$
for all $j \in \{1,\dots,N\}$, so that actually
\begin{equation*}
\|u-u_h^+\|_{\mu,\Omega}^2
=
a(u-u_h^+,\xi-I_h\xi).
\end{equation*}
We conclude the proof using Cauchy-Schwarz inequality, the fact that $\xi-I_h\xi = \xi_2-I_h\xi_2$
(since $I_h\xi_h=\xi_h$),  and the interpolation estimate \eqref{eq_interpolation_error} to get to
\begin{align*}
a(u-u_h^+,\xi-I_h\xi)
\leq&\, 
\|u-u_h^+\|_{\mu,\Omega}\|\xi-I_h\xi\|_{\mu,\Omega}
+
\beta\|(u-u_h^+)'\|_{\Omega}\|\xi-I_h\xi\|_{\Omega}\\
&\quad + \|(u-u_h^+)'\|_{\Omega}\|(\xi-I_h\xi)'\|_{\Omega}
\\
\leq&\, 
C
\left (
\mu^{1/2} h^2 \|u-u_h^+\|_{\mu,\Omega}
+
h(1+\beta h)\|(u-u_h^+)'\|_{\Omega}
\right )
\|\xi_2''\|_{\CT_h}\,,
\end{align*}
and \eqref{eq_aubin_nitsche} follows from \eqref{tmp_xi_H2}.
\end{proof}


\bibliographystyle{alpha}
\bibliography{refs.bib}

%
%

\end{document}

%% file: general_commands.tex
\usepackage{amssymb}
\usepackage{mathrsfs}

\newcommand{\eq}{:=}

\newcommand{\CP}{\mathcal P}

\newcommand{\CT}{\mathcal T}

\newcommand{\LH}{\mathscr H}

